\documentclass[a4paper]{article}
\usepackage{amsmath, amsfonts, amsthm, amssymb}
\usepackage{color}
\usepackage{tikz}
\usetikzlibrary{shapes}
\makeatletter
\let\@fnsymbol\@arabic
\makeatother
\newcommand\thankx[1]{\begingroup\let\rlap\relax\thanks{#1}\endgroup}
\newtheorem{prop}{Proposition}
\newtheorem{lemma}[prop]{Lemma}
\newtheorem{theorem}[prop]{Theorem}

\newtheorem{conjecture}[prop]{Conjecture}

\begin{document}

\title{The Robber Locating game}
\author{John Haslegrave\thankx{University of Sheffield, Sheffield, UK. {\tt j.haslegrave@cantab.net}}, Richard A. B. Johnson\thankx{University of Memphis, Memphis TN, USA. {\tt rjhnsn25@memphis.edu}}, Sebastian Koch\thankx{University of Cambridge, Cambridge, UK. {\tt sk629@cam.ac.uk}}}
\maketitle

\begin{abstract}
We consider a game in which a cop searches for a moving robber on a graph using distance probes, studied by Carragher, Choi, Delcourt, Erickson and West, which is a slight variation on one introduced by Seager. Carragher, Choi, Delcourt, Erickson and West show that for any fixed graph $G$ there is a winning strategy for the cop on the graph $G^{1/m}$, obtained by replacing each edge of $G$ by a path of length $m$, if $m$ is sufficiently large. They conjecture that the cop does not have a winning strategy on $K_n^{1/m}$ if $m<n$; we show that in fact the cop wins if and only if $m\geqslant n/2$, for all but a few small values of $n$. They also show that the robber can avoid capture on any graph of girth 3, 4 or 5, and ask whether there is any graph of girth 6 on which the cop wins. We show that there is, but that no such graph can be bipartite; in the process we give a counterexample for their conjecture that the set of graphs on which the cop wins is closed under the operation of subdividing edges. We also give a complete answer to the question of when the cop has a winning strategy on $K_{a,b}^{1/m}$.
\end{abstract}

\section{Introduction}

Pursuit and evasion games on graphs have been widely studied. Perhaps the most significant variant is the Cops and Robbers game, an instance of which is a graph $G$ together with a fixed number of cops. The cops take up positions on vertices of $G$ and a robber then starts on any unoccupied vertex. The cops and the robber take turns: the robber chooses either to remain at his current vertex or to move to any adjacent vertex, and then the cops simultaneously make moves of the same form. The game is played with perfect information, so that at any time each of the players knows the location of all others. The cops win if at any point one of them is at the same location as the robber. The cop number of a graph is the minimum number of cops required for the cops to have a winning strategy. Early results on this game include those obtained by Nowakowski and Winkler \cite{NW83}, who categorise the graphs of cop number 1, and Aigner and Fromme \cite{AF84}, who show that every planar graph has cop number at most 3. An important open problem is Meyniel's conjecture, published by Frankl \cite{Fra87}, that the cop number of any $n$-vertex connected graph is at most $O(\sqrt{n})$ -- this has been shown to be true up to a $\log(n)$ factor for random graphs by Bollob{\'a}s, Kun and Leader \cite{BLK13}, following which \L uczak and Pra\l at improved the error term \cite{LP10}. More recently, several variations on the game have been analysed by Clarke and Nowakowski (e.g. \cite{CN00}).

In this paper we consider the Robber Locating game, introduced in a slightly different form by Seager \cite{Sea12}, and further studied by Carragher, Choi, Delcourt, Erickson and West \cite{CCDEW}, in which a cop probes a vertex at each turn and is told the current distance to the robber. For ease of reading we shall refer to the cop as female and the robber as male. In this setting the cop is not on the graph herself, and can probe vertices without restriction; she wins if at any point she is able to determine the robber's current location. Clearly the cop can win eventually with probability 1 on a finite graph against a robber who has no knowledge of her future moves, simply by probing random vertices until she hits the current location of the robber. This naturally leads to a different emphasis: we consider the question of whether the cop has a strategy which is guaranteed to win in bounded time, or equivalently whether she can catch an omniscient robber. We say that a graph is \textit{locatable} if such a strategy exists. A similar game phrased in terms of a cat and mouse, in which the cat wins only if it probes the current location of the mouse, and receives no information otherwise, but the mouse must move at each turn, was recently analysed by one of the authors \cite{Has13}.

In the Robber Locating game each round consists of a move for the robber, in which he either moves to an adjacent vertex or stays where he is, followed by a probe of a particular vertex by the cop. The cop then receives a response giving the current distance of the robber from the vertex probed. She wins if she is then able to identify the robber's location. In the game as introduced by Seager there was an additional rule that the robber cannot move to the vertex probed in the previous round (the \textit{no-backtrack condition}). Carragher, Choi, Delcourt, Erickson and West consider the game without this restriction, as do we.

The authors of \cite{CCDEW} write $G^{1/m}$ for the graph obtained by replacing each edge of $G$ by a path of length $m$ through new vertices. Each such path is called a \textit{thread}, and an \textit{original vertex} in $G^{1/m}$ is a vertex which corresponds to a vertex of $G$. The main result of \cite{CCDEW} is that $G^{1/m}$ is locatable provided $m\geqslant \min\{n(G),1+\max\{\mu(G)+2^{\mu(G)},\Delta(G)\}\}$, where $\mu(G)$ is the metric dimension of $G$. The notion of metric dimension was introduced independently by Slater \cite{Sla75}, and by Harary and Melter \cite{HM76}. The metric dimension of $G$ is the size of the smallest set $S$ of vertices such that for every $x,y\in V(G)$ with $x\neq y$ there is some $z\in S$ with $d(x,z)\neq d(y,z)$. 

The authors of \cite{CCDEW} give better bounds on $m$ for complete bipartite graphs, and in this case we will find the critical value of $m$ exactly. They also conjecture that their bound is tight for complete graphs, i.e. that $K^{1/m}_n$ is locatable if and only if $m\geqslant n$. We show that in fact, except for a few small values of $n$, the actual threshold is $n/2$. They also prove that no graph of girth 3, 4 or 5 is locatable. The cycle $C_6$ is not locatable, and so they ask whether there is a locatable graph of girth 6. We give an example of such a graph, but show that no bipartite graph of girth 6 is locatable. In the process we give a counterexample to their conjecture that if $G$ is locatable then so is any graph obtained by subdividing a single edge of $G$. 

\section{Graphs of girth 6}
\label{Girth6Section}

In this section we first give an example of a locatable graph of girth 6, together with an explicit strategy for the cop. Define $H$ to be the graph obtained from the cycle $v_1v_2\cdots v_{11}$ by adding the edge $v_3v_9$. $H$ consists of a 6-cycle and a 7-cycle with an edge in common. We include an illustration of $H$ in Figure \ref{robbergirth6graph}.

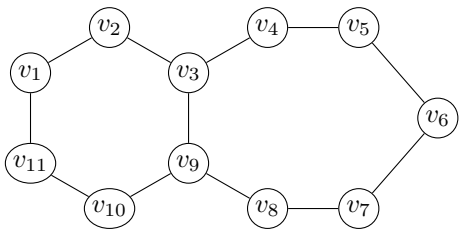
\begin{figure}[ht!] \centering
  \begin{tikzpicture}[scale=0.6]
    \tikzstyle{vertex}=[draw,shape=ellipse,minimum size=15pt,inner sep=0pt, fill=white]
		\tikzstyle{smallvertex}=[draw,shape=ellipse,minimum size=5pt,inner sep=0pt, fill=white]
% Place vertices. 
	\foreach \x/\y/\name/\alabel in {-3.464/1/$v_1$/V1, -1.732/2/$v_2$/V2, 0/1/$v_3$/V3, 1.732/2/$v_4$/V4, 3.732/2/$v_5$/V5, 5.464/0/$v_6$/V6, 3.732/-2/$v_7$/V7, 1.732/-2/$v_8$/V8, 0/-1/$v_9$/V9, -1.732/-2/$v_{10}$/V10, -3.464/-1/$v_{11}$/V11}
	{\node[vertex] (\alabel) at (\x, \y) {\name};}
% Places edges
	\foreach \start/\end in {V1/V2, V2/V3, V3/V4, V4/V5, V5/V6, V6/V7, V7/V8, V8/V9, V9/V10, V10/V11, V11/V1, V3/V9}
		{\draw (\start) -- (\end);}
  \end{tikzpicture}
	\caption{A cycle of length 6 and one of length 7 sharing an edge.}
  \label{robbergirth6graph}
\end{figure}

\begin{theorem}\label{c78}The graph $H$ as defined above is locatable.
\end{theorem}
\begin{proof}
We first give several situations from which the cop can either win or reduce to an earlier situation, and then show how she can reach a winning situation.
\begin{enumerate}
\item If the robber is known to be at $v_2$ or $v_4$ then the cop wins by probing $v_1$. 
\item If the robber is known to be at $v_3$ or $v_4$ then the cop probes $v_9$, winning or reducing to (i).
\item If the robber is known to be at $v_3$ or $v_8$ then the cop probes $v_7$, winning or reducing to (ii).
\item If the robber is known to be at $v_3$ or $v_9$ then the cop probes $v_{10}$, winning or reducing to (ii) or (iii).
\item If the robber is known to be at $v_4$ or $v_5$ then the cop wins by probing $v_6$.
\item If the robber is known to be at $v_5$ or $v_7$ then the cop probes $v_8$, winning or reducing to (v).
\item If the robber is known to be at $v_6$ or $v_8$ then the cop probes $v_6$, winning or reducing to (vi).
\item If the robber is known to be at $v_4$ or $v_7$ then the cop probes $v_7$, winning or reducing to (vii) or (ii).
\item If the robber is known to be at $v_4$ or $v_8$ then the cop probes $v_9$, winning or reducing to (iii) or (viii).
\item If the robber is known to be at $v_1$ or $v_{11}$ then the cop wins by probing $v_2$.
\end{enumerate}
The cop starts by probing $v_6$. If the answer is 0 she has won, and if it is 1, 2, 3, or 5 she has reduced to (vi), (ix), (iv), or (x) respectively. Otherwise the answer must be 4, in which case she probes $v_2$. This locates him unless the answer is 1 (when the robber must be at $v_1$ or $v_3$) or 2 (when he must be at $v_{11}$ or $v_9$). These two cases are equivalent by the symmetry of $H$, so assume the former. Now the cop probes $v_6$. If the answer is 4 the robber must be at $v_2$. If not the cop has reduced to (iv) or (x). 
\end{proof}

We have shown that there is a locatable graph of girth 6, answering a question of \cite{CCDEW}. Next we show that a significant class of graphs of girth 6 are non-locatable.

\begin{theorem}Any bipartite graph of girth 6 is non-locatable.
\end{theorem}
\begin{proof}Let $G$ be a bipartite graph of girth 6 and let $C$ be a 6-cycle of $G$. We show that the robber can win even if he is restricted to $V(C)$, by proving that if there are two non-adjacent possible robber locations in $V(C)$ after the $t^\text{th}$ probe, then no matter what vertex the cop probes next, some answer will leave two non-adjacent possible robber locations.

Suppose the robber may be at either of two non-adjacent vertices in $V(C)$ after the cop's $t^\text{th}$ probe. There are at least 5 vertices in $V(C)$ which the robber may have reached before the $(t+1)^\text{st}$ probe. Suppose the cop's $(t+1)^\text{st}$ probe is at some vertex $v$, and consider the distances from $v$ to these 5 vertices. Writing $d$ for the minimum of these distances, each one must be either $d$, $d+1$, $d+2$ or $d+3$. Since there are 5 vertices, some two must be at the same distance from $v$, so if that distance is returned there are two vertices in $V(C)$ which are possible robber locations after the $(t+1)^\text{st}$ probe. Since these two vertices are at the same distance from $v$, and $G$ is bipartite, they cannot be adjacent.
\end{proof}

Write $H'$ for the graph obtained by subdividing the edge $v_5v_6$ of $H$. Since $H'$ is bipartite (it consists of a 6-cycle and an 8-cycle with an edge in common), the robber wins on $H'$, but the cop wins on $H$ by Theorem \ref{c78}. Consequently these two graphs give a counterexample to the conjecture of \cite{CCDEW} that subdividing an edge of any cop-win graph gives another cop-win graph.

\section{Subdivisions of complete graphs}
\label{CompleteSection}

In this section we consider graphs of the form $K^{1/m}_n$. We show that if $m < n/2$ the robber wins and if $m \geqslant n/2$ for $n \geqslant 14$ the graph is locatable. For the remaining cases when $n$ is small we note the few exceptional cases that do not follow this behaviour. Consequently for each $n$ we shall have determined the winning player in all cases.

If $x$ and $y$ are original vertices of $G^{1/m}$ which correspond to adjacent vertices of $G$ we will write $x\cdots y$ for the thread of length $m$ between them. We use ``a vertex on $x\cdots y$'' to mean any of the $m+1$ vertices of the thread, but ``a vertex inside $x\cdots y$'' will exclude $x$ and $y$. When $m$ is even we will use the term ``midpoint'' for the central vertex of a thread, and when $m$ is odd we will use the term ``near-midpoint'' for either vertex of the central edge of a thread. 

We will present the proofs separately for the robber and cop winning conditions. We begin with the proofs that the robber wins for $m \leqslant (n-1)/2$, which will rely on him being able to move between original vertices without being located by the cop. 

\begin{theorem}\label{CARmsmallrobberwin}Let $m < n/2$. Then the robber wins on the graph $K_n^{1/m}$.
\end{theorem}
\begin{proof}We prove this by giving an explicit strategy for the robber that achieves the following. Assuming at some time he could be in a set of two original vertices, then we claim he can either remain in this pair of original vertices, or reach another pair without being located, and hence he can evade capture indefinitely. We will denote the set of original vertices $\{v_1, \ldots, v_n\}$.

Let us first assume that following a probe by the cop (which we will refer to as the $0^\text{th}$ probe) the robber reveals that he could be in the pair of original vertices $\{v_1, v_2\}$, but that the cop does not know which of them he is in. After this he can move to anywhere in $(v_1 \cup v_2 \cup N(v_1) \cup N(v_2))$. Firstly we will separately consider the result of the cop's first probe, which can be in one of two places. 

\begin{enumerate}
\item If her probe was equidistant to $v_1$ and $v_2$ then the robber can claim to have remained in $\{v_1, v_2\}$, and thus still be in $(v_1 \cup v_2 \cup N(v_1) \cup N(v_2))$ after the probe. If the cop always probes vertices that are equidistant from $v_1$ and $v_2$ then the robber can repeat this, evading capture indefinitely. 

\item If her probe was not equidistant to $v_1$ and $v_2$ then it was on a thread incident to at least one of them. Let us call the vertex she probes here $p$. Without loss of generality we may assume both that this probe is her first probe (ignoring any that were equidistant to $\{v_1, v_2\}$ and came before it), and that it is in the span of $v_1$. Following this probe the robber will now adopt his motivating strategy of moving towards a new original vertex. He can thus return the distance $(d(a, p) + 1)$, claiming that he was in $v_1$, and so moved to the neighbourhood of $v_1$ at the previous step. He will then continue moving down some thread towards another original vertex. Given that the robber now commits to follow this strategy the cop only needs to determine his destination before he reaches it. We will show that this is not possible by keeping a count of how many threads the cop has not yet eliminated. This first probe only eliminates the thread that $p$ is on, so following it the cop knows that the robber was at distance 1(and is now at distance 2) from $a$ and is moving along one of $(n-2)$ possible threads.
\end{enumerate}

Each subsequent probe can eliminate at most 2 threads for the robber, since probing anywhere on a thread from $v_1$ eliminates only that thread and probing inside $v_i \cdots v_j$ eliminates only $v_1 \cdots v_i$ and $v_1 \cdots v_j$. The robber can then remove those from his possible destinations and continue moving away from $a$. Hence after $t$ steps the robber is at distance $(t+1)$ from $a$ and at most $2t-1$ threads have been eliminated. After the $(m-1)^\text{st}$ step the robber reaches the remaining possible original vertices that he could have been heading towards. There were initially $n-1$ threads that he could have been heading down, and so after $(m-1)$ steps he could be on any of at least $n - 1 - (2(m-1) - 1) = n - 2m + 2 \geqslant 3$ possible threads. 

There are two possible scenarios to consider. Firstly, if as described above, the cop eliminates 2 threads on every probe except the first, then he would be unable to determine if the robber had gone halfway down a thread (pausing at the first near midpoint for a step if $m$ is odd) and then returned to $a$. Hence in this case after the $(m-1)^\text{st}$ probe the robber could move into any of at least 4 original vertices (those at either end of the uneliminated threads). If the cop did check to see if the robber turned around he would have to do so by probing on a vertex on a thread of $a$, and this would only eliminate one thread on that turn. This would mean she would eliminate one fewer thread, leaving him at least 4 threads he could be on after $(m-1)$ steps and thus at least 4 original vertices he could reach. In either case he can move into a set of at least 4 original vertices. The next probe by the cop must lie on some thread between at most 2 of them, so at least 2 will be equidistant to the next probe. The robber can now claim to have moved into that pair, and so can reach another pair of original vertices as required. Repeating this process lets him avoid capture indefinitely.
\end{proof}

We now turn our attention to the bound for the cop winning. We shall show that if $m > (n-1)/2$ then the cop can follow a simple strategy to locate the robber, which proceeds in three stages. This argument also requires $m \geqslant 7$, but that only leaves a few small cases to check manually. The second stage of this strategy works slightly differently depending on if $m$ is odd or even, but the motivating idea is the same so we present it in a single proof.

\begin{theorem}\label{CARlargemcopwin}
Let $m \geqslant n/2$ and $m \geqslant 7$. Then the cop wins on the graph $K_n^{1/m}$.
\end{theorem}
\begin{proof}
Our strategy for the cop runs in three stages. In the first stage she forces the robber to enter some original vertex, although she does not attempt to control which. In the second stage she narrows down the set of original vertices that he could be in to a set of size 2. In the final stage she locates him.

In Stage 1 the cop probes all the original vertices in any order until she either gets an answer equal to $m$ or finds two original vertices at distance less than $m$ from the robber. If she gets an answer equal to $m$ then she knows he has entered an original vertex, and moves to Stage 2. If this does not happen then he must have remained on a single thread. When probing either end of it she would get an answer less than $m$, and by noting which two original vertices this occurs on she can identify which thread he is on, and locate him. Thus either the robber is located or the cop moves to Stage 2.

In Stage 2 the cop wishes to narrow down the set of possible original vertices the robber could be in to a set of size 2. She will do this by eliminating candidates, so let us now re-order the original vertices as $v_1, \ldots, v_n$ such that $v_1$ is the last original vertex that she probed in Stage 1 -- hence the robber is known not to be in $v_1$ at the start of Stage 2.  Throughout she will track the candidates she has eliminated by maintaining a counter $r$ which is the index of the last vertex that she eliminated. Hence we set $r=1$ initially, and throughout this stage having eliminated the vertices up to $v_r$ she will be trying to eliminiate $v_{r+1}$ and thus increment $r$. We can assume throughout that $r < (n-2)$, as once she has eliminated $v_{n-2}$ there are only two vertices left, and she can proceed to Stage 3.

To eliminate $v_{r+1}$ the cop begins by probing this vertex, which can give one of five possible responses. Three of these are simple to deal with:

\begin{enumerate}
\item \emph{The distance is $0$}. The cop has found the robber and wins the game.
\item \label{CARStage2Casem-1}\emph{The distance is $m-1$}. The cop then knows that the robber was in an original vertex of higher index, and that he has left it, moving towards $v_{r+1}$. The cop can now force the robber to return to the original vertex that he came from by alternatingly probing $v_{r+1}$ and the remaining original vertices with indices higher than $r+1$ in order. If the robber moves into $v_{r+1}$ the cop will detect this and thus locate him easily, and if he does not return then she will eventually find the vertex he came from, and thus locate him. Hence he must return, which she will detect when she gets distance $m$. Along this process she will potentially eliminate not just $v_{r+1}$ but possibly many more candidates -- she proceeds by setting $r$ to the highest index that she has eliminated, and probing the next original vertex.
\item \emph{The distance is $m$}. The cop concludes that the robber is still in an original vertex of higher index than $(r+1)$. She increases $r$ by 1, and repeats the process by probing the next original vertex.
\item The most complicated case to deal with is when \emph{the distance is $m+1$}. The cop now concludes that the robber was in an original vertex of higher index, say $v_i$, and has left it moving towards another original vertex, say $v_j$. She now has two situations to consider. If $j \leqslant r$ then identifying $v_j$ before he reaches it will let her force him back into $v_i$ as in case (ii) above. If $j > r$ (and thus $j > (r+1)$ as if $j = (r+1)$ then the distance would have been $m-1$ which was case (ii) above), then she is less concerned with finding $v_j$, it suffices for her to force him into either $v_i$ or $v_j$, as then she can continue with the above process having eliminated all the original vertices up to $v_{r+1}$ as required. She will therefore address these situations sequentially.

Firstly the cop establishes whether $j > r$ by checking all the vertices in $v_1, \ldots, v_r$ to see if they are the destination for the robber. She can do this by a similar strategy to the worst case in Theorem \ref{CARmsmallrobberwin}. Ideally at each step she would check two possible destinations by probing the midpoints of the threads linking the first $r$ original vertices. In general she will not be able to do this for the first step, but she can begin by probing $v_1$ which eliminates that as a destination. If the robber announces distance $m$ then he has returned to $v_i$, and the cop can continue Stage 2 with $v_{r+1}$ eliminated. If he gives distance $m+1$ then he is still at distance 1 from $v_i$, and she can continue to probe through the set $\{v_2, \ldots, v_r\}$ until he moves in either direction or she eliminates all of them -- in the latter case we move to the next paragraph which outlines what to do once they have all been eliminated. If at some point the robber answers $m+2$ then she knows he was not heading to the vertex just probed but has moved to the second layer of vertices from $v_i$. From this point she can eliminate two vertices from $\{v_2, \ldots, v_r\}$ at each step by either probing midpoints if $m$ is even or near-midpoints if $m$ is odd. Either way she can tell whether he moves back towards $v_i$, in which case she moves back to probing single vertices once he gets back to the first layer to identify the exact moment he returns to $v_i$, or keeps eliminating pairs if he does not. If he continues to head away from $v_i$ by eliminating two vertices at each step she can eliminate $2(m-3) + 1$ before he reaches another original vertex. But as $m > (n-1)/2$ and there were only at most $(n-3)$ original vertices in $\{v_1, \ldots, v_r\}$ this leaves only two vertices that he can reach. By probing a vertex on the thread between these last two vertices she can distinguish whether he is in this pair, allowing her to move to Stage 3 if he is. Hence if he tries to move towards $\{v_1, \ldots, v_r\}$ she will either locate him, move to Stage 3 or force him back into $v_i$ successfully.

If he left the $v_i \cdots v_j$ thread either by reaching $v_i$ during the above probes or $v_j$ on the last the cop will detect this during them, allowing her to repeat Stage 2 having eliminated $v_{r+1}$. %See ending file comments here. This is one of a few situations where if $m$ is odd then the robber could technically be either in $v_i, v_j$ or a vertex adjacent to one of them on a known thread. This is however a case we can basically ignore and assume he is only in the original vertices. If he was in the interior vertex then if he uses that fact for example by moving to the second layer on his move, any probe would reveal this. Thus he cannot do that, so it is irrelevant that he could have been there.
If he has not left this thread she could then probe original vertices with indices higher than $(r+1)$ to eliminate those directly until she found either of $v_i$ or $v_j$ -- in which case she would proceed as in Case (ii) to force him into the other end of the thread, and repeat Stage 2 with more vertices eliminated. 
%There is a slight bit of funniness in the final probe in $\{v_1, v_r\}$, if $m$ was odd and this final probe gave answer $k+m$ you wouldn't know if he was in $\{v_i, v_j\}$ or adjacent to $v_i$ heading to $v_r$, but this forms part of our commented argument which we decided to omit.

\item The final case to consider is when \emph{the distance is 1}, which means the cop has found the original vertex that the robber was on, and he has moved 1 away from it. The strategy here is very similar to that above, she first makes sure that he is not moving towards $\{v_1, \ldots, v_r\}$, making sure to note if he returns to $v_i$, and then afterwards continues checking the remaining vertices in pairs. Carrying out the same analysis shows that in most cases he will be located when he returns to $v_i$, the only case when he is not is when he either moves halfway down a thread towards $\{v_{n-2}, v_{n-1}, v_n\}$ and then moves back to $v_i$ or goes all the way to $v_n$. But as this is the only case where the cop does not locate him directly if it occurs she will know, and thus be able to move to Stage 3 knowing he is in either of $v_i$ or $v_n$.
\end{enumerate}

We now move to Stage 3, which starts after the cop makes some probe and knows the robber is in one of two original vertices. We will label them as $\{a, b\}$, and note that he can move to the neighbourhood of them before the cop's first probe in Stage 3. For her first probe the cop probes the vertex at distance 1 from $a$ on the $a \cdots b$ thread. This allows her to distinguish whether the robber was on $a$ or $b$ before, and whether he is on the $a \cdots b$ thread now or another one. The cop wins immediately unless the robber answers distance 2 or distance m, in which case he has left the initial vertex he was on and moved towards an initial vertex other than $a$ or $b$. Without loss of generality we shall assume he was on $a$ and is thus now moving to one of the other $(n-2)$ possible locations, noting that this first probe reveals him to be at distance 1 from $a$.  

Her strategy now reduces to finding which thread he is on before he can reach the other end of it, being sure to note if he returns to $a$. The second probe varies according to whether $m$ is even or odd.  If $m$ is even then the cop probes a midpoint of a thread between two original vertices that have not been eliminated yet, whereas if it is odd then the cop probes a neighbour of a near-midpoint, say the vertex on the $c \cdots d$ thread that is distance 3 further from $d$ than $c$. In either case we can distinguish whether the robber is heading to that pair or not -- the one slightly complicated case is if $m$ is odd and he responds with $(m-1)/2+m$ in which case he could have remained at distance 1 from $a$ or be distance 2 from $a$ heading specifically towards $d$. If following this probe the robber uses the fact that he could have been distance 2 from $a$ to move to the vertex distance 3 from $a$ the cop will notice on her next probe and locate him. In this case the cop can therefore assume that the robber is at distance 1 and effectively eliminate $c$ from the possible destinations, doing so without him having moved closer to another original vertex so effectively for free. Thus the cop's second probe can always eliminate two possible destinations for the robber -- and by probing at midpoints if $m$ is even or near-midpoints if $m$ is odd this also holds for the subsequent probes. 

After $t$ probes he will be within distance $t$ of $a$, and she will have eliminated $1 + 2(t-1)$ possible destinations. Thus after $(m-1)$ probes there are only at most three original vertices left that he could be moving between, two possible destinations (which we shall refer to as $v$ and $w$) and $a$. Including the possibility that he turned around at the midpoint or near-midpoint (according to the parity of $m$), and assuming at each step he continued to move (as otherwise it is easier to locate him), this means that following the $(m-1)^\text{st}$ probe he is either distance $0, 1, 2, (m-2), (m-1)$ or $m$ from $a$ along either the $a \cdots v$ thread or the $a \cdots w$ thread. However, in this case he can be located by probing the vertex at distance 2 from $v$ along the $a \cdots v$ thread, provided $m \geqslant 7$, so hence he can be located even in this worst case scenario, completing the proof.
\end{proof}

This answers the question of who wins on $K_n^{1/m}$ for all but a small number of cases, which we summarise in the conclusion. We now turn our attention to bipartite graph in the following section.

\section{Subdivisions of complete bipartite graphs}
\label{BipartiteSection}

We now turn our attention to complete bipartite graphs, where we are able to determine the winning player on $K_{a,b}^{1/m}$ for any $a$, $b$, and $m$. In \cite{CCDEW} it is shown that the cop wins for $m\geqslant \max\{a,b\}$, but in fact the cop wins if and only if $m \geqslant \min \{ a,b \} - 1$, provided $a,b \geqslant 4$. Throughout this section we shall write $A$ and $B$ for the sets of original vertices in $K_{a,b}^{1/m}$ corresponding to the two vertex classes of $K_{a,b}$, with $|A|=a$ and $|B|=b$.

\begin{theorem}\label{biplower}If $a,b\geqslant 3$ and $m \leqslant \min\{a,b\}-2$ then $K_{a,b}^{1/m}$ is not locatable.
\end{theorem} 
\begin{proof}We will prove the stronger statement that the robber wins even if he is required to be at an original vertex for every $m^\text{th}$ probe, alternating between $A$ and $B$, so that he is in $A$ at the time of the $km^\text{th}$ probe for every even $k$. We show that, provided the cop has not won after the $km^\text{th}$ probe, the robber can ensure that she has not won by the $(k+1)m^\text{th}$ probe. For ease of writing, we assume that $k$ is even. 

Suppose that the robber is at $u\in A$ for the $km^\text{th}$ probe, but that the $km^\text{th}$ probe does not locate him uniquely. We show that, no matter which vertices the cop probes, there are two possible threads for the robber to travel along between the $km^\text{th}$ and $(k+1)m^\text{th}$ probes, which the cop is unable to distinguish between, so that she will not be able to win by time $(k+1)m$. Suppose her $(km+l)^\text{th}$ probe (for some $1\leqslant l\leqslant m$) is at vertex $z$, which is on the thread $x \cdots y$ for some $x\in A$ and $y\in B$. For each $v\in B$, write $w_{v,l}$ for the vertex on the thread $u\cdots v$ at distance $l$ from $u$. If $x\neq u$ then for any $v\neq y$ we have $d(z,w_{v,l})=\min\{d(z,x)+2m-l,d(z,y)+m+l\}$, whereas if $x=u$ then, again for any $v\neq y$, we have $d(z,w_{v,l})=d(z,x)+l$. Suppose that for each $l$ with $1\leqslant l \leqslant m$ the answer consistent with the robber being at any one of the vertices $w_{v,l}$ for $v\neq y$ is received from the $(km+l)^\text{th}$ probe. Then each probe eliminates at most one of the threads leaving $u$, and since $m$ probes have been made, and $m \leqslant \min\{a,b\}-2$, at least 2 remain, so the cop has not yet won.
\end{proof}

Note that if $a, b\geqslant 2$ then the robber can win on $K_{a,b}$ by ensuring he is in the opposite part to the vertex the cop probes at every time. In the case where $\min\{a, b\} = 3$, Theorem \ref{biplower} can be strengthened to say that the robber will win for $m=2$.

\begin{lemma}If $\min\{a,b\}=3$ then the graph $K_{a,b}^{1/2}$ is not locatable.
\end{lemma}
\begin{proof}Suppose that after the cop's $t^\text{th}$ probe there are two possible locations for the robber which are both in $A$ or both in $B$, say $u$ and $v$ with $u, v\in A$. We show that the robber can ensure either that there are still two possible locations, both in $A$ or both in $B$, either after the $(t+1)^\text{st}$ probe or after the $(t+2)^\text{nd}$. If the $(t+1)^\text{st}$ probe is equidistant from $u$ and $v$ this is trivial, as the robber can return the distance to $u$ or $v$. If the $(t+1)^\text{st}$ probe is $u$ or $v$ at time $t+1$ then all neighbours of $u$ will be equidistant, so the robber can claim to be at one of them. Any vertex in $A$ is equidistant from all vertices in $B$, and any other vertex is equidistant from all but one of the vertices in $B$, so no matter what vertex the cop chooses for her $(t+2)^\text{nd}$ probe, there will be at least $b-1\geqslant 2$ vertices in $B$ at the same distance from it. By this point the robber can have reached any of these without being caught. The only remaining case is for the $(t+1)^\text{st}$ probe to be at a vertex which is adjacent to either $u$ or $v$, say the vertex $w$ between $u$ and $x$ with $x\in B$. Let $y$ and $z$ be two other vertices in $B$. The midpoints of the threads $u\cdots y$, $u\cdots z$ and $v\cdots x$ are all at distance 2 from $w$, so if the robber moves to one of these the cop cannot determine which. Then no matter which vertex the cop probes at time $t+2$, some two of $x$, $y$ and $z$ are at the same distance, and so the robber can ensure there are two possible locations in $B$ after this probe.
\end{proof}

We have shown that $K_{a,b}^{1/m}$ is not locatable for $m \leqslant \min\{a,b\}-2$ when $\min\{a,b\}>3$, or for $m \leqslant \min\{a,b\}-1$ when $2 \leqslant \min\{a,b\}\leqslant 3$. Next we show that in all other cases $K_{a,b}^{1/m}$ is locatable. Note that the cop can win on the star $K_{1,b}$ by probing leaves in turn. This covers the case $\min\{a,b\}=1$. Next we deal with the case $\min\{a,b\}=2$.

\begin{lemma}The graph $K_{2,b}^{1/2}$ is locatable for any $b\ge 2$.
\end{lemma}
\begin{proof}Write $x$ and $y$ for the two vertices in $A$. Let the cop start by probing $x$. If she receives the answer 2 the robber is in $B$. If the answer is 0 or 4 she has won. If it is 1 or 3 she knows that the robber is adjacent to $x$ or $y$ respectively. 

Now we show that the cop can win from a position in which she knows that the robber is in a particular subset of the neighbourhood of $x$ (or, equivalently, if she knows the robber is in a particular subset of the neighbourhood of $y$), and she can win from a position in which she knows that the robber is in a fixed subset of $B$. We prove both simultaneously by induction on the size of the subset, $k$. In each case if $k=1$ she has already won. 

If the robber was at one of $k$ neighbours of $x$, the cop probes one of the $k$ adjacent vertices of $B$. If the answer is at most 2 then the robber is caught. If the answer is 3 then he is known to be at one of $k-1$ neighbours of $x$ and if it is 4 he is known to be at one of $k-1$ vertices of $B$; in either case we are done by induction.

If the robber was at one of $k$ vertices in $B$, the cop probes one of these. An answer of 2 is impossible, and if the answer is 1 then she can win by next probing $x$. If the answer is 4 then she knows the robber is at one of $k-1$ vertices of $B$, and we are done by induction. If the answer is 3 then she probes $x$ next; now if the answer is 0 or 4 she has won, and if it is 1, 2, or 3 she has reduced to one of $k-1$ vertices adjacent to $x$, in $B$, or adjacent to $y$ respectively, so we are done by induction. 
\end{proof}

Finally we show that, provided $m \geqslant 3$, the only non-locatable graphs of the form $K_{a,b}^{1/m}$ are those given in Theorem \ref{biplower}.

\begin{theorem}Let $a,b\geqslant 3$. If $m \geqslant \min \{a,b\}-1$ and $m\geqslant 3$ then $K_{a,b}^{1/m}$ is locatable.
\end{theorem} 
\begin{proof}Suppose $a\leqslant b$. Again we give a two-stage winning cop strategy. In the first stage we show that the cop can win or establish that the robber is in $B$, and in the second stage we show that she may win once he knows that the robber is in $B$.

In the first stage, the cop probes vertices in $A$ in turn until she receives an answer of $m$ (indicating that the robber is in $B$) or less than $m$. This must eventually happen, since if the robber does not reach $B$ he must remain nearer one particular vertex in $A$ than any other, and when the cop probes this vertex she will get an answer of less than $m$. In this case write $x$ for the vertex in question. Once the cop has found $x$, the robber cannot leave his current thread without moving either to $x$ or to some vertex in $B$, so the cop then probes vertices in $B$ until she receives an answer of $2m$ (indicating that the robber is in $B$) or at most $m$ (in which case she can determine his location).

In the second stage we show that the cop may win from a position where the robber is known to be in a fixed subset of $B$, by induction on the size of the subset, $k$. This is true for $k=1$ as she has already won. If $k>1$ then write $B'$ for the set of $k$ vertices in question. The cop starts by probing the vertex adjacent to $B'$ on the thread $x\cdots y$ for some $x\in A$ and $y\in B'$. The possible answers are 0 (if the robber is at that vertex), 1 (if he is at $y$), 2 (if he is at some other neighbour of $y$), $2m-2$ (if he is on another thread leading to $x$), $2m-1$ (if he is at a vertex of $B'$ other than $y$), and $2m$ (if he is on a thread which does not include $x$ or $y$). Since $m\geqslant 3$, these are all different. An answer of 0 or 1 is an immediate win for the cop, and after an answer of $2m-1$ she wins by the induction hypothesis. After an answer of $2m-2$ the cop probes vertices of $B'$ until either she receives an answer of at most $m$, winning, or she receives an answer of $2m$, in which case she knows the robber is at one of at most $k-2$ vertices of $B'$ and she wins by the induction hypothesis. After an answer of $2$ or $2m$, the robber must be in a thread which does not reach $x$. The cop now probes vertices of $A$, other than $x$, in turn, until she receives an answer of $m$, $2m$, or less than $m$. One of these must eventually happen since either the robber reaches one end of the thread he is currently on, or he remains in the same thread until such time as the cop probes its end in $A$. If the answer $2m$ occurs first, the cop knows that the robber has reached some vertex $u\in A$ which is neither $x$ nor one she has probed since the robber left $B$. Since the robber has taken at least $m$ steps to reach $A$, she has probed at least $m-1$ vertices in $A$, and together with $x$ she has eliminated at least $m\geqslant a-1$ vertices of $A$, so there is only one possibility and the robber is caught. If the answer $m$ occurs first then the robber is at a vertex of $B'$, and, since the cop knows whether or not this is $y$, she has either caught the robber or reduced to a set of $k-1$ vertices, so wins by the induction hypothesis. If an answer less than $m$ occurs first, say when probing $u$, then the cop has won if that answer is 0, or if the robber was initially known to be on a thread meeting $y$. Otherwise, she knows that the robber is on some thread $u\cdots v$ for $v\in B'\setminus\{y\}$; now she proceeds by probing vertices of $B'\setminus\{y\}$ in turn until she receives an answer of at most $m$ (in which case she has won) or of $2m$ (in which case she knows that the robber is at some vertex in $B'\setminus\{y\}$, and so wins by induction).
\end{proof}

We can now completely determine which graphs of the form $K^{1/m}_{a,b}$ are locatable. If $\min\{a,b\}\geqslant 4$ then $K^{1/m}_{a,b}$ is locatable if and only if $m\geqslant \min\{a,b\}-1$, whereas if $\min\{a,b\}\leqslant 3$ then $K^{1/m}_{a,b}$ is locatable if and only if $m\geqslant \min\{a,b\}$.

\section{Conclusion and Open Problems}

We note that in the proof of Theorem \ref{CARlargemcopwin} the condition that $m \geqslant 7$ is only required for the final part of Stage 3, which only arises when $m = n/2$. If $m \geqslant n/2 + 1$ then this condition is not necessary, and the result still holds that the cop wins for all values of $n$. This answers the question of who wins for which $m$ on $K_n^{1/m}$ in almost all cases, except for a few small values. These are small enough to be checked by hand, we note that almost all of them obey the same relationship of the cop winning if $m \geqslant n/2$ and the robber winning if $m < n/2$, with the exceptions that the robber can also win in the following cases: $m=2, n=3$ or $4$; $m=3, n=6$; and $m=5, n=10$. 

In Section \ref{Girth6Section} we have shown that it is not necessarily true that subdividing a single edge of a locatable graph yields another locatable graph. It remains an open conjecture that it is however true that subdividing every edge of a locatable graph yields another locatable graph. Another natural question is whether for every graph $G$ there is some $m_G$ for which $G^{1/m}$ is locatable if and only if $m\geqslant m_G$. We have shown that this is the case for complete graphs and complete bipartite graphs by finding exact values of $m_G$ in those cases in Sections \ref{CompleteSection} and \ref{BipartiteSection} respectively. The question remains open in generality, although the authors believe it to be true.

Moreover, the authors note that the ideas developed in Section \ref{BipartiteSection} can also be used to prove similar results for complete $ r $-partite graphs. More precisely, if $G$ is such a graph with parts of sizes $ a_{1}, \ldots, a_{r} $, where $ a_{1} \leqslant a_{2} \leqslant \cdots \leqslant a_{r} $, then the lengths of subdivision required and sufficient to make $G^{1/m}$ locatable are each about $ \max \{ (a_{1} + \cdots + a_{r-1})/2 , a_{r-1} \} $. This generalises the results on complete bipartite graphs. In the case of balanced $r$-partite graphs this gives a threshold of $(n/2)(1-1/r) + O(1)$.

Finally, we note that for all graphs considered in this paper, subdividing the edges of an $ n $-vertex-graph about $ n/2 $ times is sufficient to make it locatable, and we conjecture that this is indeed the case for all finite graphs.

\begin{conjecture}
For all sufficiently large $n$, if $ G $ is a graph on $ n $ vertices then $ G^{1/m} $ is locatable for every $ m \geqslant n/2 $.
\end{conjecture}

\section{Acknowledgements}
The first author acknowledges support from the European Union through funding under FP7--ICT--2011--8 project HIERATIC (316705), and is grateful to Douglas B. West for drawing his attention to this problem. The second author acnowledges support through funding from NSF grant DMS~1301614 and MULTIPLEX grant no. 317532, and is grateful to the organisers of the $8^\text{th}$ Graduate Student Combinatorics Conference at the University of Illinois at Urbana-Champaign for drawing his attention to the problem. The third author acknowledges support through funding from the European Union under grant EP/J500380/1 as well as from the Studienstiftung des Deutschen Volkes. The second and third authors would also like to thank Yuval Peres and the Theory Group at Microsoft Research Redmond for hosting them while some of this research was conducted.

\end{document}